\documentclass[11pt,leqno]{amsart}

\usepackage{amsmath,amsthm}
\usepackage {latexsym}
\usepackage{amssymb}

\newcommand{\sign}{\mbox{\rm sign}}
\newcommand{\eps}{{\varepsilon}}

\newcommand{\R}{{\mathbb R}}

\newcommand{\1}{{\mathbf 1}}

\newcommand{\les}{\lesssim}

\newcommand{\txt}{\textstyle}
\newcommand{\scr}{\scriptstyle}

\newcommand{\Rplus}{R^+}
\newcommand{\Rminus}{R^-}

\newcommand{\la}{\langle}
\newcommand{\ra}{\rangle}
\def\norm[#1][#2]{\|#1\|_{#2}}
\def\bignorm[#1][#2]{\big\|#1\big\|_{#2}}
\def\Bignorm[#1][#2]{\Big\|#1\Big\|_{#2}}
\def\japanese[#1]{\langle #1 \rangle}
\def\Im[#1]{{\rm Im}(#1)}
\def\Re[#1]{{\rm Re}(#1)}

\newtheorem{theorem}{Theorem}
\newtheorem{lemma}[theorem]{Lemma}

\theoremstyle{remark}
\newtheorem{remark}[theorem]{Remark}

\begin{document}

\title[The Schr\"odinger equation with a non-smooth magnetic potential]
{Strichartz estimates for Schr\"odinger operators with a non-smooth
magnetic potential}
\date{March 28, 2008}

\author{Michael\ Goldberg}
\thanks{The author received partial support from NSF grant DMS-0600925 during
the preparation of this work.}
\address{Department of Mathematics, Johns Hopkins University,
3400 N. Charles St., Baltimore, MD 21218}
\email{mikeg@math.jhu.edu}

\begin{abstract}
We prove Strichartz estimates for the absolutely continuous evolution of a
Schr\"odinger operator $H = (i\nabla + A)^2 + V$ in $\R^n$, $n \ge 3$.
Both the magnetic and electric potentials are time-independent and satisfy
pointwise polynomial decay bounds.  The vector potential $A(x)$ is assumed to
be continuous but need not possess any Sobolev regularity.  This work is a 
refinement of previous methods, which required extra conditions on
${\rm div}\,A$ or $|\nabla|^{\frac12}A$ in order to place the first order
part of the perturbation within a suitable class of pseudo-differential
operators.
\end{abstract}

\maketitle

\section{Background}
This paper continues an investigation into the dispersive properties of
Schr\"odinger operators taking the form
\begin{equation}\label{eq:H}
\begin{aligned}
H = (i\nabla + A(x))^2 + V(x) &= -\Delta + i(A \cdot \nabla + \nabla\cdot A)
 + |A|^2 + V\\
&= \Delta + L(x)
\end{aligned}
\end{equation}
where $A(x)$ is a vector field in $\R^n$, $n \ge 3$, 
and $V(x)$ is a scalar function. The associated evolution equation
\begin{equation} \label{eq:Schr}
\begin{cases}
iu_t(t,x) = Hu(t,x)\\
u(0,x) = u_0(x)
\end{cases} 
\end{equation}
models the motion of a single charged particle within an ambient
electromagnetic field, with $V$ and $A$ serving respectively as the
electrostatic and magnetic potentials.  Such operators also appear routinely
when a nonlinear Schr\"odinger equation is linearized around a nonzero static
solution.

The starting point for estimates in the free case ($L \equiv 0$) is an
explicit formula for the propagator $e^{it\Delta}$ based on Fourier inversion.
\begin{equation*}
u(t,x) = (4\pi it)^{-n/2}\int_{R^n} e^{i\frac{|x-y|^2}{4t}}u_0(x)\,dx
\end{equation*}
It is immediately clear that $u = e^{it\Delta}u_0$ satisfies a family of
dispersive bounds
\begin{equation} \label{eq:dispersive}
\norm[u(t,\,\cdot\,)][q] \les |t|^{-n(\frac12-\frac1q)}\norm[u_0][q'], \qquad
2 \le q \le \infty,
\end{equation}
the $p=2$ case following from Plancherel's identity and the $p=\infty$ case
from evaluation of the convolution integral above.  Each of the dispersive
bounds with $2 \le q \le \frac{2n}{n-2}$ can be incorporated into a $TT^*$
argument to prove one of the Strichartz inequalities,
\begin{equation} \label{eq:freeStrich}
\norm[e^{it\Delta}u_0][L^p_tL^q_x] \les \norm[u_0][2], \qquad
{\txt \frac2p + \frac{n}{q} = \frac{n}{2}}, \quad 2 \le p,q \le \infty.
\end{equation}
This line of argument requires nothing more than the Hardy-Littlewood-Sobolev
inequality (except for the $p=2$ endpoint case which is more
delicate~\cite{KeT98}), and the resulting estimate have played a central 
and the resulting estimates have played a central role in the well-posedness
theory of nonlinear Schr\"odinger equations from their inception.
It would therefore be desirable to prove an analogous statement for solutions 
of~\eqref{eq:Schr} with more general $H$.  Unfortunately, it may not suffice
to replace the Laplacian in~\eqref{eq:Schr} with another Schr\"odinger
operator $H$ because of the possible existence of point spectrum.
Each eigenvalue of $H$ gives rise to a solution of the form $u(t,x) = 
e^{-i\lambda t}\psi(x)$, negating the bound
\begin{equation*}
\norm[e^{-itH}u_0][L^p_tL^q_x] \les \norm[u_0][2]
\end{equation*}
for every $p < \infty$.

For short range self-adjoint perturbations of the Laplacian, the spectrum of
$H$ should remain absolutely continuous along the positive real axis, with
discrete negative eigenvalues possibly accumulating at zero.  No eigenvalues
are embedded within the continuous spectrum~\cite{KoT06}.  In many cases it
is possible to prove that the linear Schr\"odinger evolution decomposes into a 
discrete sum of bound states, plus a radiation term that enjoys
the same dispersive properties as a free wave.  Our goal is to expand the
class of potentials for which this behavior is known to occur.

Following the notation in Chapter XIV of~\cite{H85}, define the
function space
\begin{equation} \label{eq:B}
\norm[f][B] := \sum_{j=0}^\infty 2^{\frac{j}{2}}\norm[f][L^2(D_j)],
\qquad
\norm[f][B^*] := \sup_{j\ge0}\,2^{-\frac{j}{2}}\norm[f][L^2(D_j)],
\end{equation}
where $D_j = \{x\in R^n: |x| \sim 2^j\}$ is a decomposition of $\R^n$ into
dyadic shells for $j \ge 1$ and $D_0 = \{|x|\in\R^n: |x| \le 1\}$.  These
share numerous characteristics with weighted $L^2$, and we designate $B_s$
to be the associated homogeneous Sobolev spaces with norm
\begin{equation} \label{eq:B1}
\norm[f][B_{s}] := \norm[\,|\nabla|^{s} f][B], \qquad 
\norm[f][B_{s}^*] := \norm[\,|\nabla|^s f][B^*].
\end{equation}
One other
Banach space related to $B$ will also play an important role during the
discussion.  Let
\begin{equation} \label{eq:X}
\norm[g][X] := \Big(\sum_{j=0}^\infty
     2^{3j}\norm[g][L^\infty(D_j)]^2\Big)^{1/2}
\end{equation}
Observe that pointwise multiplication by a function in $X$ maps $B^*$ to the
weighted space $\japanese[x]^{-1}L^2(\R^n)$, and also maps $\japanese[x]L^2$
to $B$.  

\begin{theorem}\label{thm:main}
Let $H = (i\nabla + A)^2 + V$ be a magnetic Schr\"odinger operator on ${\R^n}$,
$n \ge 3$, whose scalar and magnetic potentials are bounded functions
satisfy the conditions
\begin{align}
&\japanese[x]^2V \in L^\infty(\R^n),\ \ {\rm and}\ \ 
\limsup_{|x|\to\infty} |x|^2|V(x)| = 0 \tag{C1} \label{eq:Vdecay}\\
&A\ {\rm is\ continuous,}\ \ {\rm and}\ \ \norm[A][X] < \infty \tag{C2} 
\label{eq:Adecay}
\end{align}
Let $P_{ac}(H)$ represent the orthogonal projection onto the absolutely
continuous spectrum of $H$.  If zero is not an eigenvalue or a resonance of
$H$, then the Strichartz inequalities
\begin{equation} \label{eq:Strichartz}
\norm[e^{-itH}P_{ac}(H)u_0][L^p_tL^q_x] \les \norm[u_0][2], \qquad
{\txt \frac2p + \frac{n}{q} = \frac{n}{2}}
\end{equation}
are valid for each exponent $p > 2$, as well as the Kato smoothing bound
\begin{equation} \label{eq:Kato}
\norm[g(x)|\nabla|^{1/2}e^{-itH}P_{ac}(H)u_0][L^2_tL^2_x]
\les \norm[g^2][X]^{\frac12} \norm[u_0][2].
\end{equation}
\end{theorem}

It is not possible to follow precisely in the footsteps of the free case.
In particular, no replacement for \eqref{eq:dispersive} is known in the case
$A \not= 0$, and to achieve the full range of exponents likely requires
additional regularity assumptions~\cite{GV06},~\cite{CCV08}.  There are two
alternative approaches to this problem, both of which attempt to confine
perturbative arguments to the more forgiving $L^2$ setting.

The first method is rooted in pseudo-differential calculus, decomposing the
solution into wave packets and finding suitable parametrices to describe
their respective trajectories.  This techniques were applied to
magnetic Schr\"odinger operators as early as~\cite{R92}.  More recent advances
have highlighted the flexibility to work on manifolds other than $\R^n$
(as in.~\cite{BT06}) and with time-varying coefficients~\cite{MMT07}.

We instead adopt the framework outlined in~\cite{RS04}, where Strichartz
estimates follow directly from a pair of Kato smoothing bounds, one each for
the free and perturbed propagators. The starting point is to express the
perturbed evolution according to Duhamel's formula, with $H = -\Delta + L$.
\begin{equation} \label{eq:Duhamel}
e^{-itH}u_0 = e^{it\Delta}u_0 - i\int_0^te^{i(t-s)\Delta}Le^{-isH}u_0\,ds
\end{equation}
The free evolution term is controlled by~\eqref{eq:freeStrich}.  To prove
a Strichartz estimate for the integral term, it typically suffices to find a
factorization 
\begin{equation*}
L = \sum_{j=1}^J Y_j^*Z_j
\end{equation*}
for which $Z_j$ is a smooth perturbation relative to $H$ on its absolutely
continuous subspace
and $Y_j$ is smooth relative to the Laplacian.  Following Kato's~\cite{K65}
theory of smoothing, this reduces to proving uniform estimates for the
resolvents of both $H$ and $-\Delta$.  

It is well known that $Y_j$ must not have
order greater than $\frac12$, and one expects the same to be true of $Z_j$.
The natural factorization of $L = A\cdot \nabla$ would therefore 
distribute half of a derivative to each of $Y_j$ and $Z_j$. 
Some regularity of the coefficients $A(x)$ is needed in order to commute
half a derivative across them.

To handle less smooth magnetic potentials, we take advantage of the parabolic
nature of the Schr\"odinger equation to replace powers of the gradient with
derivatives in the time direction.  Since $A(x)$ is time-independent, it
commutes with operators such as $\japanese[\partial_t]$ without 
apparent difficulty.  There is a price to be paid later for this convenience:
some of the precise structure of Duhamel's formula (i.e. $0<s<t$) is lost
in the process.  Another apparent sacrifice is that the polynomial decay
condition~\eqref{eq:Adecay} falls one-half a power short of the natural
scaling law $A(x) \sim |x|^{-1}$.

In the next section we choose a factorization for the first-order
perturbation found in~\eqref{eq:H} and prove preliminary smoothing bounds
for each term.  The task is not complicated, as many of the underlying
resolvent estimates appear in the literature in precisely the form needed.
Section~\ref{sec:Proof} contains the direct proof of Theorem~\ref{thm:main}.
The key ingredient is a set of modified smoothing estimates whose form is
dictated by the domain of integration in~\eqref{eq:Duhamel}.  
The technical issues
raised by our use of fractional time-derivatives are addressed here.

\section{Resolvent and Smoothing Estimates}

The magnetic and scalar potentials collectively perturb the Laplacian with
a differential operator of the form
\begin{align*}
L &= i(A \cdot \nabla + \nabla \cdot A) + |A|^2 + V \\
  &= \sum_{j=1}^3 Y_j^*Z_j
\end{align*}
where our factorization of choice is to take
\begin{equation} \label{eq:Y*Z}
\begin{aligned}
&Y_1 = \japanese[x]^{-1}\japanese[\partial_t]^{\frac14}, \ \ &
&Y_2 = i\japanese[x]\japanese[\partial_t]^{-\frac14}A\cdot \nabla, \ \ &
&Y_3 = \big|V +|A|^2\big|^{1/2}, \\
&Z_1 = Y_2, &
&Z_2 = Y_1, &
&Z_3 = \sign(V + |A|^2)Y_3.
\end{aligned}
\end{equation}
Operators $Y_1$ and $Y_2$ produce vector-valued functions, and could be
further split (if necessary) into their three coordinate directions. 
The first set of smoothing estimates for $Y_j$ relative to the Laplacian are
now easy to obtain.
\begin{lemma} \label{lem:H_0smooth}
Suppose $V(x)$ and $A(x)$ are bounded and 
satisfy~\eqref{eq:Vdecay},~\eqref{eq:Adecay}. 
Each of the operators $Y_j$, $j = 1, 2, 3$, satisfies the bound
\begin{equation} \label{eq:H_0smooth}
\norm[Y_je^{it\Delta}f][L^2(\R\times\R^n)] \les \norm[f][2]
\end{equation}
for all functions $f \in L^2(\R^n)$.  To prevent ambiguity, the first two
of these mapping estimates are interpreted as follows.
\begin{align}
\int_{\R^n} \japanese[x]^{-2} 
\bignorm[e^{it\Delta}f(\,\cdot\,,x)][H^{\frac14}(\R)]^2 \,dx
&\les \norm[f][2]^2 \\
\int_{\R^n} \japanese[x]^{2} 
\bignorm[A\cdot \nabla \big(e^{it\Delta}f\big)(\,\cdot\,,x)
][H^{-\frac14}(\R)]^2 \,dx &\les \norm[A][X]^2 \norm[f][2]^2
\end{align}
with the intermediate function $e^{it\Delta}f$ taking values at all times
$t \in \R$, both positive and negative.
\end{lemma}
\begin{proof}
Let $\lambda$ represent the Fourier variable dual to $t$.  After setting up the
$TT^*$ operator and taking the Fourier transform in $T$, 
the smoothing bounds for $Y_1$, $Y_2$, $Y_3$ are 
essentially equivalent to these  
properties of the free resolvent.
\begin{equation} \label{eq:freeresbounds} 
\begin{aligned}
\bignorm[\japanese[x]^{-1}\big(\Rplus_0(\lambda)-\Rminus_0(\lambda)\big)
  \japanese[x]^{-1}f][2] &\les \japanese[\lambda]^{-\frac12}\norm[f][2] \\
\bignorm[\japanese[x]g\nabla 
   \big(\Rplus_0(\lambda)-\Rminus_0(\lambda)\big) 
   \nabla \japanese[x]gf][2] 
   &\les \japanese[\lambda]^{\frac12} \norm[g][X]^2 \norm[f][2] \\
\end{aligned}
\end{equation}
uniformly over $\lambda \ge 0$ with any exponent $\sigma \ge 1$.  
We have adopted the shorthand notation for resolvents
\begin{equation} \label{eq:resolventdef}
\begin{aligned}
R_L^\pm(\lambda) &= \lim_{\eps\downarrow 0}\big(H-(\lambda\pm i\eps)\big)^{-1}
\\
R_0^\pm(\lambda) &= \lim_{\eps\downarrow 0}
  \big(-\Delta - (\lambda\pm i\eps)\big)^{-1}
\end{aligned}
\end{equation}
The first 
inequality can be found in~\cite{A75}, and a particularly sharp version
appears in \cite{S92}.  The second is a standard result in~\cite{H85},
together with the observation that $\Delta R_0^\pm(\lambda) = 
-I -\lambda R_0^\pm(\lambda)$.
\end{proof}

In order to show that each $Z_j$ is a smooth perturbation relative to 
(the absolutely continuous part of) $H$,
it suffices to verify that the resolvent estimates in~\eqref{eq:freeresbounds}
continue to hold when
$R_0^\pm(\lambda)$ is replaced by $R_L^\pm(\lambda)$.  It is convenient to
break the problem into three separate regimes according to whether
$\lambda \ll 1$, $\lambda \sim 1$, or $\lambda \gg 1$.  The latter two cases
have been considered at length elsewhere, so we are content to collect
these results into a single statement.
\begin{lemma} \label{lem:midhighenergy}
Let $A$ be a vector field and $V$ a scalar function 
satisfying conditions~\eqref{eq:Vdecay}-\eqref{eq:Adecay}.
Given any number $\lambda_0 > 0$, there exists a constant 
$C(L,\lambda_0)<\infty$ such that
\begin{equation} \label{eq:midhighenergy}
\begin{aligned}
\bignorm[\japanese[x]^{-1}\big(\Rplus_L(\lambda)-\Rminus_L(\lambda)\big)
  \japanese[x]^{-1}f][2] 
 &\le C(L,\lambda_0) \japanese[\lambda]^{-\frac12}\norm[f][2] \\
\bignorm[\japanese[x]g \nabla
   \big(\Rplus_L(\lambda)-\Rminus_L(\lambda)\big)
   \nabla \japanese[x]gf][2]
   &\le C(L,\lambda_0) \japanese[\lambda]^{\frac12}\norm[g][X]^2\norm[f][2]
\end{aligned}
\end{equation}
uniformly over all $\lambda > \lambda_0$.
\end{lemma}

\begin{proof}
According to Proposition~4.3 of~\cite{EGS07}, there exists a number 
$\lambda_1(L) < \infty$ such that the perturbed resolvents $R_L^\pm(\lambda)$
satisfy the bounds
\begin{align*} 
\norm[R_L^\pm(\lambda)][B\to B^*] &\le C\lambda^{-\frac12}\\
\norm[R_L^\pm(\lambda)][B_{-1}\to B_1^*] &\le C\lambda^{\frac12}
\end{align*}
for all $\lambda > \lambda_1$.  A uniform bound over the interval
$\lambda \in [\lambda_0, \lambda_1]$ is proved as part of Theorem~1.3 
in~\cite{IS06}.  It is assumed there that no eigenvalues are embedded in
the positive half-line, a condition which was subsequently shown 
in~\cite{KoT06} to hold for the entire class of potentials under consideration.
\end{proof}

\begin{remark}
The mapping bound from $B$ to $B^*$ is considerably stronger than what is
needed to prove Lemma~\ref{lem:midhighenergy}.  Alternatively, by projecting
onto the range $\lambda \in [\lambda_0,\infty)$ rather than the entire
continuous spectrum, one can obtain smoothing estimates while assuming 
less decay of the potentials.  It should suffice to let $A,V$ belong to
the space
\begin{equation*}
\tilde{X} := \Big\{g \in L^\infty: \sum_{j=0}^\infty\norm[g][L^\infty(D_j)]
 < \infty \Big\}.
\end{equation*}
\end{remark}

The remaining short interval $\lambda \in [0,\lambda_1]$ is handled by a
compactness argument.  By applying resolvent identities, the difference
$\Rplus_L(\lambda)-\Rminus_L(\lambda)$ can  be expressed as
\begin{equation*}
\Rplus_L(\lambda)-\Rminus_L(\lambda) =
(I + \Rplus_0(\lambda)L)^{-1}\big(\Rplus_0(\lambda)-\Rminus_0(\lambda)\Big)
(I + L\Rminus_0(\lambda))^{-1}
\end{equation*}
It is already established (see~\eqref{eq:freeresbounds}) that the
difference of free resolvents maps the space
$Y := \japanese[x]^{-1}L^2(\R^n) + B_{-1}$, equipped with the norm
\begin{equation} \label{eq:Y}
\norm[f][Y] = \inf\Big\{\norm[\japanese[x]f_1][2] + 
\norm[f_2][B_{-1}]:\ f_1 + f_2= f \Big\}
\end{equation} 
to its dual $Y^* = \japanese[x]L^2 \cap B_1^*$ with uniformly
bounded operator norm.  We therefore need only to show that
$(I + \Rplus_0(\lambda)L)^{-1}$ exists as a uniformly bounded family of
operators on $Y^*$ over this range of $\lambda$.  The same will be true of
the other inverse by duality.

\begin{lemma} \label{lem:lowenergy}
Let $V$ and $A$ satisfy conditions~\eqref{eq:Vdecay}-\eqref{eq:Adecay}
and suppose that the Schr\"odinger operator $H$ does not have an eigenvalue
or resonance at zero.  Then
\begin{equation*}
\sup_{\lambda \in [0,\lambda_0]} 
\bignorm[(I + \Rplus_0(\lambda)L)^{-1}][Y^*\to Y^*] < \infty.
\end{equation*}
\end{lemma}
\begin{proof}
Any potential $V \in \japanese[x]^{-2}L^\infty$ satisfying~\eqref{eq:Vdecay}
or $A \in X$ can be approximated in the appropriate norm by a function
with compact support.  As a consequence, $L$ acts as a compact operator taking
$Y^*$ to $Y$.  Meanwhile, the resolvent 
$R_0^\pm(\lambda)$ maps $Y$ back to $Y^*$.  The end result is that 
$I + \Rplus_0(\lambda)L$ is always a compact perturbation of the identity.

The Fredholm Alternative asserts that $(I + \Rplus_0(\lambda)L)$ must possess
a bounded inverse unless there there is a nontrivial null-space consisting
of functions $f \in Y^*$ that solve the equation $f = -\Rplus_0(\lambda)Lf$.
Any such $f$ would also be a distributional solution of the eigenfunction
equation
\begin{equation*}
-\Delta f + Lf = \lambda f.
\end{equation*}
The combined efforts of~\cite{KoT06} and~\cite{IS06} rule out their existence
for each $\lambda >0$.  To avoid them when $\lambda = 0$ we must make
include an {\it a priori} assumption
that zero is not an eigenvalue or resonance of $H$.  The need to exclude
resonances stems from the fact that $Y^*$ includes functions that do not
decay rapidly enough to belong to $L^2(\R^n)$.  

\begin{remark}
In~\cite{JK79} the presence of an eigenvalue at $\lambda = 0$ is shown to
nullify some Strichartz estimates even after projecting away from the
associated eigenvector with $P_{ac}(H)$.
\end{remark}

Now that $(I + \Rplus_0(\lambda)L)^{-1}$ exists at each $\lambda \in
[0,\lambda_0]$, the uniform bound on their norms follows from continuity
with respect to $\lambda$. More precisely, the resolvents $\Rplus_0(\lambda)$
enjoy the weak-* continuity property described below.   
Given two functions $f, g \in Y$ and $\Lambda \in [0,\infty)$,
\begin{equation*}
\lim_{\lambda\to\Lambda} \la \Rplus_0(\lambda)f, g\ra 
  = \la \Rplus_0(\Lambda)f, g\ra.
\end{equation*}
The main idea here is again that $f$ and $g$ can be approximated by compactly
supported functions, and that the integration kernel of 
$\Rplus_0(\Lambda) - \Rplus_0(\lambda)$ vanishes pointwise on bounded sets.

Now suppose there were a sequence of values $\lambda_n$ for which
the norm of $(I + \Rplus_0(\lambda_n)L)^{-1}$ grew without bound.  This would
imply the existence of unit functions $f_n \in Y^*$ for which
$\norm[f_n + \Rplus_0(\lambda_n)f_n][Y^*] \to 0$.  By passing to a subsequence
we may assume that $\lambda_n \to \Lambda$ and that $f_n$ converges in the
weak-* topology to a function $f_\infty\in Y^*$.

Recall that $L$ is a compact operator, which means that 
$Lf_n$ should converge to $Lf_\infty \in Y$ in norm.
By the weak-*continuity of resolvents, it follows that 
$\Rplus_0(\Lambda)Lf_\infty$ is the weak-* limit of $\Rplus_0(\lambda_n)Lf_n$.
This sequence is sufficiently close to $-f_n$ that
$-\Rplus_0(\Lambda)Lf_\infty$ must be a weak-* limit of $f_n$
This makes $f_\infty$ a solution to the eigenfunction equation, so
$f_\infty \equiv 0$.

The contradiction arises because now we have $\norm[Lf_n][Y] \to 0$,
which would cause $\norm[f_n + \Rplus_0(\lambda_n)f_n][Y^*] \to 1$,
violating the part of the construction where
$\norm[f_n + \Rplus_0(\lambda_n)Lf_n][] \to 0$.
\end{proof}

These resolvent estimates can now be combined to prove the second Kato
smoothing estimate for our factorization of $L$.
\begin{lemma} \label{lem:Hsmooth}
Suppose $V(x)$ and $A(x)$ are bounded and
satisfy~\eqref{eq:Vdecay},~\eqref{eq:Adecay}.
Each of the operators $Z_j$, $j = 1, 2, 3$, satisfies the bound
\begin{equation} \label{eq:Hsmooth}
\norm[Z_je^{-itH}P_{ac}(H)f][L^2(\R\times\R^n)] \les \norm[f][2]
\end{equation}
for all functions $f \in L^2(\R^n)$.  To prevent ambiguity, the first two
of these mapping estimates are interpreted as follows.
\begin{align*}
\int_{\R^n} \japanese[x]^{-2}
\bignorm[e^{-itH}P_{ac}(H)f(\,\cdot\,,x)][H^{\frac14}(\R)]^2 \,dx
&\les \norm[f][2]^2 \\
\int_{\R^n} \japanese[x]^{2}
\bignorm[A\cdot \nabla \big(e^{-itH}P_{ac}(H)f\big)(\,\cdot\,,x)
][H^{-\frac14}(\R)]^2 \,dx &\les \norm[A][X]^2 \norm[f][2]^2
\end{align*}
with the intermediate function $e^{-itH}P_{ac}(H)f$ taking values at all times
$t \in \R$, both positive and negative.
\end{lemma}
The Kato smoothing conclusion~\eqref{eq:Kato} in Theorem~\ref{thm:main}
is a direct interpolation of the two inequalities above.

\section{Proof of Theorem~\ref{thm:main}} \label{sec:Proof}

The intended purpose of Lemmas~\ref{lem:H_0smooth} and~\ref{lem:Hsmooth} was
to control the integral term in Duhamel's formula~\eqref{eq:Duhamel}.  For
each term in the factorization of $L$, one would like to
apply~\eqref{eq:Hsmooth} to obtain a function in $L^2(\R\times\R^n)$,
followed by the dual form of~\eqref{eq:H_0smooth}
to return it to $L^2(\R^n)$, with the free Strichartz 
estimate~\eqref{eq:freeStrich} then giving a map into $L^p_tL^q_x$.

The domain of integration in~\eqref{eq:Duhamel} is clearly stated as
$0 \le s \le t$. Unfortunately, the interpretation of fractional derivatives
and time-propagation in both~\eqref{eq:H_0smooth} and~\eqref{eq:Hsmooth}
is just as clearly two-sided in $s$.  The desired Strichartz 
estimate~\eqref{eq:Strichartz} therefore follows from two
modified smoothing bounds which include a sharp cutoff to preserve the
triangular region of integration.

\begin{lemma} \label{lem:newHsmooth}
Suppose $V(x)$ and $A(x)$ are bounded and
satisfy~\eqref{eq:Vdecay},~\eqref{eq:Adecay}.
Each of the operators $Z_j$, $j = 1, 2, 3$, satisfies the bound
\begin{equation} \label{eq:newHsmooth}
\bignorm[Z_j \big(\1_{\{s\ge 0\}} e^{-isH}\big)P_{ac}(H)f][L^2_sL^2_x]
\les \norm[f][2]
\end{equation}
for all functions $f \in L^2(\R^n)$. 
\end{lemma}

\begin{lemma} \label{lem:newH_0smooth}
Under the same hypotheses, each of the operators $Y_j$, $j=1,2,3$, satisfies
a bound
\begin{equation} \label{eq:newH_0smooth}
\Bignorm[\int_\R \big(\1_{\{t\ge s\}} e^{i(t-s)\Delta}\big)
 Y_j^*g(s,x)\,ds][L^p_tL^q_x] \les \norm[g][L^2_sL^2_x]
\end{equation}
for each Strichartz pair $\frac{2}{p} + \frac{n}{q} = \frac{n}{2}$, \ $p > 2$,
\ and every $g \in L^2(\R\times\R^n)$.
\end{lemma}

\begin{proof}[Proof of Lemma~\ref{lem:newHsmooth}]
For $Z_3$ this is a trivial consequence of~\eqref{eq:Hsmooth}.  For both
$Z_1$ and $Z_2$ it is also an immediate consequence, as pointwise
multiplication by $\1_{\{s\ge 0\}}$ is a bounded operator on every Sobolev
space $H^\alpha$, $|\alpha| < \frac12$.
\end{proof}

\begin{proof}[Proof of Lemma~\ref{lem:newH_0smooth}]
For the pointwise multiplication operator $Y_3$, the argument in~\cite{RS04}
requires no modification.  The dual form of~\eqref{eq:H_0smooth} combines
with~\eqref{eq:freeStrich} to yield a bound
\begin{equation*}
\Bignorm[\int_\R e^{i(t-s)\Delta}Y_3g(s,x)\,ds][L^p_tL^q_x] 
   \les \norm[g][L^2_sL^2_x]
\end{equation*}
for all Strichartz pairs $(p,q)$ including the endpoint $(2, \frac{2n}{n-2})$.
The Christ-Kiselev lemma~\cite{CK01} (in particular the version appearing 
in~\cite{SS00}) confirms that the same bound is true if the integral is only
taken over the triangular region $t \ge s$, provided there is a strict
inequality $p > 2$ in the exponents.

When the same logic is applied to $Y_1$, the first integral estimate takes the
form
\begin{equation*}
\Bignorm[\int_\R \big(\japanese[\partial_s]^{\frac14}e^{i(t-s)\Delta}\big)
 \japanese[x]^{-1}g(s,x)\,ds][L^p_tL^q_x]  \les \norm[g][L^2_sL^2_x].
\end{equation*}
Directly applying the Christ-Kiselev lemma leads to the further bound
\begin{equation*}
\Bignorm[\int_\R \1_{\{t\ge s\}}
  \big(\japanese[\partial_s]^{\frac14}e^{i(t-s)\Delta}\big)
  \japanese[x]^{-1}g(s,x)\,ds][L^p_tL^q_x]  \les \norm[g][L^2_sL^2_x].
\end{equation*}
The statement of~\eqref{eq:newH_0smooth} for $Y_1$ instead concerns the
operator inequality
\begin{equation*}
\Bignorm[\int_\R \japanese[\partial_s]^{\frac14}
 \big(\1_{\{t\ge s\}} e^{i(t-s)\Delta}\big)
  \japanese[x]^{-1}g(s,x)\,ds][L^p_tL^q_x]  \les \norm[g][L^2_sL^2_x]
\end{equation*}
so we are left to bound the difference between the two.  This consists of
two terms,
\begin{multline*}
\int_\R \Big(\1_{\{t<s\}} \japanese[\partial_s]^{\frac14}
 \big(\1_{\{t\ge s\}}e^{i(t-s)\Delta}\big)  \\
   - \1_{\{t\ge s\}} \japanese[\partial_s]^{\frac14}
 \big(\1_{\{t<s\}} e^{i(t-s)\Delta}\big)\Big)\japanese[x]^{-1}g(s,x)\,ds.
\end{multline*}
It suffices to control the first term, as the second will behave identically
by symmetry.  After the initial multiplication by $\japanese[x]^{-1}$, the
rest of the operator is a convolution in $\R \times \R^n$ with a kernel
$K_1(t,x)$ satisfying the bounds
\begin{equation} \label{eq:K1}
|K_1(t,x)| \les  \begin{cases} 
|x|^{-n-\frac12}\japanese[t|x|^{-2}]^{-\frac54}, \ &{\rm if}\ t \le 1, \\
e^{-t}\,|x|^{-n+2}\japanese[x]^{-\frac52}, \ &{\rm if}\ t>1.
\end{cases}
\end{equation}
A direct computation then shows that $K_1$ belongs comfortably to both
$L^1_tL^{n/2}_x$ and $L^2_tL^1_x$, thus convolution against $K_1$
maps $L^2(\R\times\R^n)$ (even without the extra weight $\japanese[x]^{-1}$)
to any of the mixed-norm Strichartz spaces.

The analysis for $Y_2$ is quite similar.  Once again there is a bulk term
that can be controlled by combining the dual form of~\eqref{eq:H_0smooth}
with~\eqref{eq:freeStrich} and the Christ-Kiselev lemma.  There is a 
nonzero remainder generated by commuting multiplication by the cutoff
$\1_{t\ge s}$ with fractional integration in $s$.  In this case its precise
form is
\begin{multline*}
\int_\R \Big(\1_{\{t<s\}} \japanese[\partial_s]^{-\frac14}
 \big(\1_{\{t\ge s\}}\nabla e^{i(t-s)\Delta}\big)  \\
   - \1_{\{t\ge s\}} \japanese[\partial_s]^{-\frac14}
 \big(\1_{\{t<s\}}\nabla e^{i(t-s)\Delta}\big)\Big)
    \cdot A(x)\japanese[x]g(s,x)\,ds.
\end{multline*}
As before, multiplication by $A\japanese[x] \les 1$ is composed with a
convolution operator whose kernel $K_2(t,x)$ satisfies the bounds
\begin{equation} \label{eq:K2}
|K_2(t,x)| \les  \begin{cases}
|x|^{-n-\frac12}\japanese[t|x|^{-2}]^{-\frac34}, \ &{\rm if}\ t \le 1, \\
e^{-t}\,|x|^{-n+1}\japanese[x]^{-\frac32}, \ &{\rm if}\ t>1.
\end{cases}
\end{equation}
The kernel $K_2$ belongs to both $L^2_tL^1_x$ and $L^1_tL^{n/2,\infty}$,
which is still sufficient to generate a bounded map from $L^2_tL^2_x$ to
each of the Strichartz spaces.

The only task remaining is to verify the asserted size estimates~\eqref{eq:K1}
and~\eqref{eq:K2}.  Each one involves a fractional derivative or integral,
so they are best handled via the Fourier transform.  Recall that the 
convolution kernel for the Schr\"odinger propagator can be expressed as
\begin{equation*}
\1_{\{t\ge 0\}}e^{it\Delta}{\scr(x)} = \int_\R e^{-it\lambda}\Rminus_0(\lambda)
{\scr(x)}\,d\lambda
\end{equation*}
therefore the kernel $K_1(t,x)$ is defined to be
\begin{equation}
K_1(t,x) = \1_{\{t<0\}} \int_\R \japanese[\lambda]^{\frac14}e^{-it\lambda}
  \Rminus(\lambda){\scr(x)}\, d\lambda.
\end{equation}
For any $t<0$, the free resolvent $\Rminus(\lambda)$ has an analytic
continuation in $\lambda$ to the lower halfplane.  Moreover, its kernel has the
asymptotic bound
\begin{equation*}
|D^\alpha \Rminus(\lambda){\scr(x)}| \les |x|^{-n+2-|\alpha|}
 \,e^{|x|\Im[\lambda^{1/2}]}
\japanese[|x|\lambda^{\frac12}]^{\frac{n-3+2|\alpha|}{2}}
\end{equation*}
for derivatives of order $|\alpha| = 0,1$.

The smooth power function
$\japanese[\lambda]^{\frac14}$ is holomorphic in the halfplane with a 
segment removed extending from $-i$ downward along the imaginary axis.
This allows the contour of integration to be moved from the real axis
to the two sides of this imaginary segment.  To be precise, the original
Fourier transform existed only in the distributional sense, so we must
first mollify the behavior at infinity with a function such as 
$\frac{ki}{ki-\lambda}$ and take limits as $k\to\infty$.

On this new path, change variables so that $\lambda = -i\mu^2$.  The
result is that
\begin{align*}
|K_1(t,x)| &= C \Big|\int_1^\infty \mu (1-\mu^4)^{\frac18}e^{-t\mu^2}
 \Rminus(-i\mu^2){\scr(x)}\,d\mu \Big|\\
& \les |x|^{-n+2} \int_1^\infty \mu^{\frac32} e^{-t\mu^2}e^{-c|x|\mu}
 \japanese[|x|\mu]^{\frac{n-3}{2}}\,d\mu \\
|K_2(t,x)| &= C \Big|\int_1^\infty \mu (1-\mu^4)^{-\frac18}e^{-t\mu^2}
 \nabla \Rminus(-i\mu^2){\scr(x)}\,d\mu \Big|\\
& \les |x|^{-n+1} \int_1^\infty \mu^{\frac12} e^{-t\mu^2}e^{-c|x|\mu}
 \japanese[|x|\mu]^{\frac{n-1}{2}}\,d\mu
\end{align*}
where $c = \sqrt{1/2}$ is a fixed constant.  One possible bound comes from
dominating the factor $e^{-c|x|\mu}\japanese[|x|\mu]^\beta$ by a constant,
which leaves
\begin{align*}
|K_1(t,x)| &\les |x|^{-n+2} \int_1^\infty \mu^{\frac32}e^{-t\mu^2}\,d\mu
 \les \min\big(t^{-\frac54}, e^{-t}\big)|x|^{-n+2} \\
|K_2(t,x)| &\les |x|^{-n+1} \int_1^\infty \mu^{\frac12}e^{-t\mu^2}\,d\mu
 \les \min\big(t^{-\frac34}, e^{-t}\big)|x|^{-n+1}
\end{align*}
These appear to be optimal when $|x|^2 < t < 1$, and are adequate for our
purposes when $|x| < 1 < t$.  Another alternative is to dominate $e^{-t\mu^2}$
by $e^{-t}$, in which case
\begin{align*}
|K_1(t,x)| &\les e^{-t}|x|^{-n+2}\int_1^\infty \mu^{\frac32}
 e^{-c|x|\mu}\japanese[|x|\mu]^{\frac{n-3}{2}}\,d\mu 
  \les e^{-t}|x|^{-n-\frac12} \\
|K_2(t,x)| &\les e^{-t}|x|^{-n+1}\int_1^\infty \mu^{\frac12}
 e^{-c|x|\mu}\japanese[|x|\mu]^{\frac{n-1}{2}}\,d\mu
  \les e^{-t}|x|^{-n-\frac12}
\end{align*}
These bounds appear to be sharp when $t < |x|^2 < 1$, but are also sufficient
to complete the argument whenever $t, |x| \ge 1$.

\end{proof}

\end{document}